\documentclass[a4paper,reqno]{amsart}

\usepackage{amssymb}
\usepackage{amstext}
\usepackage{amsmath}
\usepackage{amsthm}
\usepackage{amsfonts}
\usepackage{enumerate}
\usepackage{graphicx}
\usepackage{latexsym}
\usepackage{mathrsfs}
\usepackage{mathtools}
\usepackage{paralist}
\usepackage{xcolor}
\usepackage[colorlinks=true,linkcolor=blue!55!black,citecolor=blue!55!black,urlcolor=blue!55!black]{hyperref}

\hypersetup{
  pdftitle={Self-Orthogonal Tau-Tilting Modules and Tilting Modules II: Annihilator Separation},
  pdfauthor={Xiaojin Zhang},
  pdfsubject={Annihilators and torsion separation for self-orthogonal tau-tilting modules},
  pdfkeywords={tau-tilting module, self-orthogonal module, annihilator, torsion class, triangular matrix algebra, one-point extension}
}

\newtheorem{theorem}{Theorem}[section]
\newtheorem{corollary}[theorem]{Corollary}
\newtheorem{lemma}[theorem]{Lemma}
\newtheorem{proposition}[theorem]{Proposition}

\newtheorem{question}[theorem]{Question}
\theoremstyle{definition}
\newtheorem{definition}[theorem]{Definition}
\newtheorem{remark}[theorem]{Remark}
\newtheorem{example}[theorem]{Example}
\newtheorem*{theorema}{Theorem A}
\newtheorem*{theoremb}{Theorem B}
\newtheorem*{theoremc}{Theorem C}
\newtheorem*{theoremd}{Theorem D}

\newcommand{\Ext}{\operatorname{Ext}\nolimits}
\newcommand{\Hom}{\operatorname{Hom}\nolimits}
\newcommand{\End}{\operatorname{End}\nolimits}
\newcommand{\Ann}{\operatorname{Ann}\nolimits}
\newcommand{\rad}{\operatorname{rad}\nolimits}
\newcommand{\soc}{\operatorname{soc}\nolimits}
\newcommand{\topm}{\operatorname{top}\nolimits}
\newcommand{\supp}{\operatorname{supp}\nolimits}
\newcommand{\add}{\mathsf{add}\hspace{.01in}}
\newcommand{\Fac}{\mathsf{Fac}\hspace{.01in}}
\newcommand{\Filt}{\mathsf{Filt}\hspace{.01in}}
\newcommand{\Sub}{\mathsf{Sub}\hspace{.01in}}
\newcommand{\Tors}{\mathsf{Tors}\hspace{.01in}}
\newcommand{\Sim}{\mathsf{Sim}\hspace{.01in}}
\newcommand{\modA}{\mathsf{mod}\hspace{.01in}A}
\newcommand{\findim}{\operatorname{findim}\nolimits}
\newcommand{\perpall}[1]{{}^{\perp_{\geq 0}}#1}

\numberwithin{equation}{section}

\begin{document}

\title[Annihilator separation]{Self-Orthogonal Tau-Tilting Modules and Tilting Modules II: Annihilator Separation}

\author{Xiaojin Zhang}
\address{X. Zhang: School of Mathematics and Statistics, Jiangsu Normal
University, Xuzhou, 221116, P. R. China}
\email{xjzhang@jsnu.edu.cn, xjzhangmaths@163.com}

\begin{abstract}
Let $A$ be an Artin algebra and let $T$ be a self-orthogonal
$\tau$-tilting right $A$-module. Set $I=\Ann_A(T)$. We prove that
\[
       \Hom_A(I,T)=0=\Ext_A^n(I,T)\qquad(n\geq 1).
\]
It follows that the torsion class generated by $I$ is Hom-orthogonal
to $\Sub T$. This separation yields faithfulness criteria expressed
through ideal tops, socles, and projective supports. We also prove a
finitistic-dimension obstruction: if
$B=\End_A(T)^{\rm op}$ has finite little finitistic dimension, then
$\Fac T\cap{}^{\perp_{\geq0}}T=\{0\}$. As an application, the
self-orthogonal $\tau$-tilting conjecture holds for
radical square zero Artin algebras.
Finally, the support criteria apply to semisimple-source triangular
matrix algebras with arbitrary local terminal blocks.
\end{abstract}

\maketitle

\section{Introduction}

Let $A$ be an Artin algebra. A basic $\tau$-tilting $A$-module is
sincere, but it need not be faithful. Its annihilator measures the
precise gap: a $\tau$-tilting module $T$ is a $1$-tilting
$A$-module if and only if $\Ann_A(T)=0$. Moreover, $T$ is always a
$1$-tilting module over $A/\Ann_A(T)$
\cite[Lemma~2.4 and Proposition~2.5]{CLZZ}; see also \cite{AIR}. Recall that a module $T$ is \emph{self-orthogonal} if
$\Ext_A^n(T,T)=0$ for all $n\geq 1$.
This makes the following problem natural.

\medskip
\noindent
\emph{{\bf Question}: When does self-orthogonality of a $\tau$-tilting module force
faithfulness?}
\medskip

Several homological approaches to this problem are known. Zhang proved
that a $\tau$-tilting module $T$ is $1$-tilting precisely when
$\Ext_A^i(T,\Fac T)=0$ for all $i\geq1$; in particular, finite
projective dimension together with self-orthogonality forces tilting
\cite{Zhang}. Related vanishing characterizations for sincere
silting modules over perfect rings were obtained in \cite{LiuWei}.
More comparison criteria use second extensions,
$\operatorname{Tor}_1$-vanishing for the annihilator, and delooping
levels \cite{CLZZ}. Recently, the author combined $\tau$-tilting modules with $\tau^{-1}$-tilting modules and showed that self-orthogonal $\tau$-tilting-$\tau^{-1}$-tilting modules are tilting modules \cite{Zhang2}.

In fact, we conjectured that a self-orthogonal $\tau$-tilting module is a tilting module \cite{Zhang}, which is called the self-orthogonal $\tau$-tilting conjecture in \cite{CLZZ}. Moreover, it is shown that the conjecture holds for algebras of finite global dimension \cite{Zhang}, gentle algebras \cite{C}, Gorenstein CM-finite algebras \cite{LyW}, algebras of finite representation type and minimal representation infinite algebras \cite{CLZZ}. In general, the conjecture is open.

The present paper develops a contravariant approach:
the annihilator is placed in the first variable of $\Hom_A(-,T)$ and
$\Ext_A^n(-,T)$. The resulting degree-zero vanishing is what makes
the torsion-theoretic separation below possible.

Write $I=\Ann_A(T)$. Our starting point is the following complete
orthogonality statement.

\begin{theorema}
If $T$ is a self-orthogonal and $\tau$-tilting $A$-module, then
\[
                  \Hom_A(I,T)=0
       \quad\text{and}\quad
                  \Ext_A^n(I,T)=0\quad(n\geq1).
\]
Consequently,
\[
 I\in\Sub T\quad\Longleftrightarrow\quad I=0
 \quad\Longleftrightarrow\quad T\text{ is }1\text{-tilting}.
\]
\end{theorema}

Thus the condition $\Ann_A(T)\in\Sub T$ should be viewed as a target,
not as an independent weakening of faithfulness. The degree-zero
vanishing has a stronger categorical consequence. Let
\[
       \Tors_A(I):=\Filt(\Fac I)
\]
be the smallest torsion class containing $I$. We prove the separation
formula
\begin{equation}\label{eq:intro-separation}
       \Hom_A\bigl(\Tors_A(I),\Sub T\bigr)=0,
       \qquad
       \Tors_A(I)\cap\Sub T=\{0\}.
\end{equation}
In particular, if $I\neq0$, then no simple quotient of $I$ can occur
in $\soc T$.

This observation leads to a criterion which does not require knowing
$I$ in advance. Let $\Sim(A)$ be the set of isomorphism classes of
simple right $A$-modules and put
\[
 \mathfrak T(A)=
 \bigl\{\supp\topm(J_A)\mid
 0\neq J\triangleleft A\text{ is a nilpotent two-sided ideal}\bigr\}.
\]
A subset $\Sigma\subseteq\Sim(A)$ is called an \emph{ideal-top
hitting set} if $\Sigma\cap U\neq\varnothing$ for every
$U\in\mathfrak T(A)$. Although $A$ may have infinitely many ideals,
$\mathfrak T(A)$ has at most $2^{|\Sim(A)|}-1$ members.

\begin{theoremb}
Let $T$ be a self-orthogonal $\tau$-tilting $A$-module. Assume that
\[
       \Tors_A(J)\cap\Sub T\neq\{0\}
\]
for every nonzero nilpotent two-sided ideal $J$ of $A$. Then $T$ is
a $1$-tilting module. In particular, it is enough that
$\supp\soc T$ be an ideal-top hitting set.
\end{theoremb}

As a consequence of Theorem A and Theorem B, we have the following theorem which gives a partial answer to the self-orthogonal $\tau$-tilting conjecture.

\begin{theoremc}
Let $A$ be a radical square zero Artin algebra. 
Then every self-orthogonal $\tau$-tilting $A$-module is a classical
$1$-tilting module.
\end{theoremc}


We also give applications to semisimple-source triangular
matrix algebras whose terminal diagonal blocks are arbitrary local
algebras and whose connecting bimodule is unrestricted. 

\begin{theoremd}
Assume that $A$ is Morita equivalent to
\[
 A'=\begin{pmatrix}
       S&M\\
       0&C_1\times\cdots\times C_s
    \end{pmatrix},
\]
where $S$ is a semisimple Artin algebra, each $C_j$ is a local Artin
algebra, and ${}_SM_{C_1\times\cdots\times C_s}$ has finite length on
both sides. Then every
self-orthogonal $\tau$-tilting $A$-module is $1$-tilting.
\end{theoremd}

The algebras in Theorem D include arbitrary source one-point extensions of local algebras. We also give an explicit
$51$-dimensional example which is connected, representation-infinite,
CM-infinite, of infinite global dimension, non-self-injective, and
non-gentle. The familiar commutative and one-simple-block
consequences remain as specializations rather than the main examples.
Existing work on triangular matrix algebras concerns $\tau$-tilting finiteness and the construction of support
$\tau$-tilting modules under one-point extension
\cite{AiharaHonma,GaoXie,Suarez}. Our application instead concerns
automatic faithfulness under self-orthogonality.

The paper is organized as follows. In Section 2, we combine the
canonical approximation of a $\tau$-tilting module with
self-orthogonality and establish complete annihilator orthogonality
and torsion separation. In Section 3, we develop the transversal,
ideal-top, cogeneration, and dimension-shift criteria. In Section 4, we get a
finitistic-dimension obstruction and give applications to radical square zero algebras. In Section 5,
we prove the projective-support criterion and apply it to triangular
matrix algebras and one-point extensions. We also give the explicit
$51$-dimensional family.

Throughout this paper, $A$ is an Artin $R$-algebra, where $R$ is a
commutative Artin ring, and all modules are finitely generated right
modules. We write $\modA$ for the category of such modules and
$D=\Hom_R(-,E(R/J(R)))$ for the standard duality, where $J(R)$ is the
Jacobson radical of $R$. The Auslander--Reiten translation is denoted
by $\tau$.

\section{Annihilator orthogonality and torsion separation}

In this section, we give the new theorem on complete annihilator orthogonality and its applications.

For $M\in\modA$, let $\add M$ denote the direct summands of finite direct
sums of copies of $M$. The classes $\Fac M$ and $\Sub M$ consist,
respectively, of factor modules and submodules of objects in $\add M$.
If $\mathcal C$ is a class of modules, then $\Filt(\mathcal C)$ is
the class of modules admitting a finite filtration with factors in
$\mathcal C$.  We use
\[
           \Ann_A(M)=\{a\in A\mid Ma=0\}.
\]
For a finite-length module $X$, its support $\supp X$ is the set of
isomorphism classes of its simple composition factors.

We use the standard definitions of $\tau$-rigid and $\tau$-tilting
modules from \cite{AIR}.  In particular, if $T$ is $\tau$-tilting,
there is an exact sequence
\begin{equation}\label{eq:canonical-approximation}
 A\xrightarrow{f}T_0\longrightarrow T_1\longrightarrow0,
 \qquad T_0,T_1\in\add T,
\end{equation}
in which $f$ is a left $\add T$-approximation
\cite[Proposition~2.23]{AIR}.

\begin{lemma}\label{lem:kernel}
For the approximation in \eqref{eq:canonical-approximation}, one has
\[
                         \ker f=\Ann_A(T).
\]
\end{lemma}

\begin{proof}
Set $I=\Ann_A(T)$.  Since $T_0\in\add T$ and
$f(a)=f(1)a$, one has $f(I)=0$. Conversely, for $t\in T$, the map
$h_t:A\to T$ defined by $h_t(a)=ta$ factors through $f$.  If
$a\in\ker f$, then $ta=h_t(a)=0$ for every $t\in T$, so $a\in I$.
\end{proof}

We use two further standard facts. If $T$ is $\tau$-tilting and
$I=\Ann_A(T)$, then $I$ is nilpotent and $T$ is a faithful
$\tau$-tilting, hence $1$-tilting, module over $A/I$. In particular,
$T$ is $1$-tilting over $A$ precisely when $I=0$; see
\cite[Lemma~2.4 and Proposition~2.5]{CLZZ}.

The canonical approximation controls all contravariant extensions
from the annihilator to $T$.

\begin{theorem}[Complete annihilator orthogonality]
\label{thm:orthogonality}
Let $T$ be a self-orthogonal $\tau$-tilting $A$-module and set
$I=\Ann_A(T)$.  Then
\[
              \Hom_A(I,T)=0
              \quad\text{and}\quad
              \Ext_A^n(I,T)=0\quad(n\geq1).
\]
Equivalently, $I\in\perpall{T}$.
\end{theorem}

\begin{proof}
Choose \eqref{eq:canonical-approximation} and put
$C=\operatorname{Im}f$.  It gives short exact sequences
\[
 0\longrightarrow C\longrightarrow T_0\longrightarrow T_1
 \longrightarrow0
\]
and
\[
 0\longrightarrow I\longrightarrow A\longrightarrow C
 \longrightarrow0.
\]
Because $T_0,T_1\in\add T$ and $T$ is self-orthogonal, the first
sequence yields
\[
                    \Ext_A^n(C,T)=0\qquad(n\geq1).
\]
The second sequence then gives $\Ext_A^n(I,T)=0$ for every $n\geq1$.

Applying $\Hom_A(-,T)$ to the second sequence also shows that the
restriction map
\[
                 \Hom_A(A,T)\longrightarrow\Hom_A(I,T)
\]
is surjective. Every map $A\to T$ factors through the left
$\add T$-approximation $f$ and therefore vanishes on $I=\ker f$.
The restriction map is thus zero, and $\Hom_A(I,T)=0$.
\end{proof}

\begin{corollary}\label{cor:sub-equivalence}
Under the hypotheses of Theorem \ref{thm:orthogonality}, the following
conditions are equivalent:
\begin{enumerate}[\rm(i)]
\item $I\in\Sub T$;
\item $I=0$;
\item $T$ is a $1$-tilting $A$-module.
\end{enumerate}
\end{corollary}

\begin{proof}
If $I\hookrightarrow T^m$, then Theorem \ref{thm:orthogonality}
forces this monomorphism to be zero, whence $I=0$. The converse is
immediate. The equivalence of (ii) and (iii) is the faithful
$\tau$-tilting criterion.
\end{proof}

\begin{remark}\label{rem:selforth-essential-preview}
The vanishing in degree zero uses self-orthogonality through the
surjectivity of the restriction map. It fails for general
$\tau$-tilting modules. In Example \ref{ex:selforth-essential}, a
$\tau$-tilting module $T$ satisfies
$0\neq\Ann_A(T)\in\add T$ but is not self-orthogonal.
\end{remark}

For $M\in\modA$, the smallest torsion class containing $M$ is
\[
                       \Tors_A(M)=\Filt(\Fac M).
\]
Complete orthogonality separates the torsion class generated by the
annihilator from the submodule category cogenerated by $T$.

\begin{theorem}[Torsion separation]\label{thm:torsion-separation}
Let $T$ be a self-orthogonal $\tau$-tilting $A$-module and put
$I=\Ann_A(T)$.  Then
\[
       \Hom_A\bigl(\Tors_A(I),\Sub T\bigr)=0.
\]
In particular,
\[
                     \Tors_A(I)\cap\Sub T=\{0\}.
\]
\end{theorem}

\begin{proof}
Let $X\in\Fac I$ and $Y\in\Sub T$. Choose an epimorphism
$I^r\twoheadrightarrow X$ and a monomorphism $Y\hookrightarrow T^s$.
For any map $g:X\to Y$, the composite
\[
              I^r\longrightarrow X\xrightarrow{g}Y
              \longrightarrow T^s
\]
is zero by Theorem \ref{thm:orthogonality}. The first map is
epimorphic and the last is monomorphic, so $g=0$. For fixed $Y$, the
class $\{X\mid\Hom_A(X,Y)=0\}$ is extension closed. Induction on a
$\Fac I$-filtration therefore gives the asserted vanishing on
$\Tors_A(I)$. If a module belongs to both classes, its identity map
is zero.
\end{proof}

The first visible layer of this separation is especially useful.

\begin{corollary}\label{cor:top-socle}
With $T$ and $I$ as above,
\[
              \supp\topm(I_A)\cap\supp\soc(T_A)=\varnothing.
\]
\end{corollary}

\begin{proof}
If a simple module $S$ belonged to the intersection, then $S$ would
be a quotient of $I$ and a submodule of $T$. Thus
$S\in\Tors_A(I)\cap\Sub T$, contrary to Theorem
\ref{thm:torsion-separation}.
\end{proof}

We end this section with the example showing that self-orthogonality can not be removed
from Corollary \ref{cor:sub-equivalence}.

\begin{example}
\label{ex:selforth-essential}
Let $k$ be a field and
\[
 A=kQ/J^2,
 \qquad
 Q:\quad 1\xrightarrow{\alpha}2\xrightarrow{\beta}1,
\]
where $J$ is the arrow ideal. Work with right modules, put
$P_i=e_iA$, and let $S_i=P_i/\rad P_i$.  Consider
\[
                         T=P_1\oplus S_1. 
\]
Then
\begin{enumerate}

\item  $T=P_1\oplus S_1$ is a $\tau$-tilting module.

\item A direct calculation gives $\Ann_A(T)=k\beta\simeq S_1$
Thus $0\neq\Ann_A(T)\in\add T$.  

\item One has $\Ext_A^2(S_1,S_1)\neq0$, which implies that $T$ is not self-orthogonal and $\Ann_A(T)\not\in \perpall{T}$.

\end{enumerate}
\end{example}

\section{Faithfulness criteria and nilpotent ideals}

In this section, we give several faithfulness criteria in terms of non-zero nilpotent two-sided ideals.

We begin with the following theorem on torsion-transversal detection.

\begin{theorem}[Torsion-transversal detection]
\label{thm:torsion-transversal}
Let $T$ be a self-orthogonal $\tau$-tilting $A$-module. Assume that
\begin{equation}\label{eq:torsion-transversal}
       \Tors_A(J)\cap\Sub T\neq\{0\}
\end{equation}
for every nonzero nilpotent two-sided ideal $J$ of $A$. Then $T$ is
a $1$-tilting $A$-module.
\end{theorem}

\begin{proof}
Put $I=\Ann_A(T)$. If $I\neq0$, then $I$ is a nonzero nilpotent
two-sided ideal. Condition \eqref{eq:torsion-transversal}, applied to
$J=I$, contradicts Theorem \ref{thm:torsion-separation}. Hence
$I=0$, and Corollary \ref{cor:sub-equivalence} applies.
\end{proof}

The canonical approximation gives a second algebra-level version of
the transversal condition. If $I=\Ann_A(T)$, then
\eqref{eq:canonical-approximation} induces a monomorphism
\begin{equation}\label{eq:regular-quotient-embedding}
                         A/I\lhook\joinrel\longrightarrow T_0.
\end{equation}
Thus the regular $A/I$-module, and every module cogenerated by it, is
already cogenerated by $T$.

\begin{theorem}
\label{thm:regular-quotient-transversal}
Assume that
\begin{equation}\label{eq:regular-quotient-transversal}
 \Tors_A(J)\cap\Sub(A/J)\neq\{0\}
\end{equation}
for every nonzero nilpotent two-sided ideal $J$ of $A$. Then every
self-orthogonal $\tau$-tilting $A$-module is $1$-tilting.
\end{theorem}

\begin{proof}
Let $T$ be self-orthogonal and $\tau$-tilting, and put
$I=\Ann_A(T)$. If $I\neq0$, then $I$ is nilpotent and
\eqref{eq:regular-quotient-transversal} supplies a nonzero module in
$\Tors_A(I)\cap\Sub(A/I)$. By
\eqref{eq:regular-quotient-embedding},
$\Sub(A/I)\subseteq\Sub T$. This contradicts Theorem
\ref{thm:torsion-separation}. Therefore $I=0$.
\end{proof}

\begin{corollary}
\label{cor:quotient-socle-detection}
If
\[
 \supp\topm(J_A)\cap\supp\soc((A/J)_A)\neq\varnothing
\]
for every nonzero nilpotent two-sided ideal $J$ of $A$, then every
self-orthogonal $\tau$-tilting $A$-module is $1$-tilting.
\end{corollary}

\begin{proof}
A simple module in the displayed intersection is a quotient of $J$
and a submodule of $A/J$. It is therefore a nonzero object of the
intersection in \eqref{eq:regular-quotient-transversal}.
\end{proof}

Top--socle separation makes condition
\eqref{eq:torsion-transversal} accessible through a finite set of
simple modules.

\begin{definition}\label{def:hitting}
Define the \emph{nilpotent ideal-top family} of $A$ by
\[
 \mathfrak T(A)=
 \bigl\{\supp\topm(J_A)\mid
 0\neq J\triangleleft A\text{ is nilpotent}\bigr\}.
\]
A subset $\Sigma\subseteq\Sim(A)$ is an \emph{ideal-top hitting set}
if $\Sigma\cap U\neq\varnothing$ for every $U\in\mathfrak T(A)$.
\end{definition}

\begin{corollary}[Ideal-top detection]\label{cor:hitting}
Let $T$ be a self-orthogonal $\tau$-tilting $A$-module.  If
$\supp\soc T$ is an ideal-top hitting set, then $T$ is a
$1$-tilting $A$-module.
\end{corollary}

\begin{proof}
Let $0\neq J\triangleleft A$ be nilpotent.  The hitting assumption
provides a simple module
$S\in\supp\topm J\cap\supp\soc T$.  Then $S$ is both a quotient of
$J$ and a submodule of $T$, so
\[
              0\neq S\in\Tors_A(J)\cap\Sub T.
\]
Theorem \ref{thm:torsion-transversal} now applies.
\end{proof}

\begin{corollary}[Socle sincerity]\label{cor:socle-sincere}
If $T$ is a self-orthogonal $\tau$-tilting module and $\soc T$ is
sincere, then $T$ is a $1$-tilting module.
\end{corollary}

For a self-orthogonal $\tau$-tilting module, the sufficient
conditions obtained so far form the hierarchy
\[
\begin{aligned}
 \soc T\text{ sincere}
 &\Longrightarrow \supp\soc T\text{ ideal-top hitting}\\
 &\Longrightarrow
 \Tors_A(J)\cap\Sub T\neq\{0\}
 \text{ for every }0\neq J\triangleleft A\text{ nilpotent}\\
 &\Longrightarrow T\text{ is }1\text{-tilting}.
\end{aligned}
\]
The first implication is strict by Example
\ref{ex:hitting-not-sincere}.

\begin{remark}\label{rem:finite-hitting}
Definition \ref{def:hitting} packages ideal data into a finite
set-system on $\Sim(A)$. It is enough to retain the inclusion-minimal
members of $\mathfrak T(A)$. Thus Corollary \ref{cor:hitting} can be
tested by a finite transversal problem even when the lattice of
two-sided ideals itself is infinite.
\end{remark}

\begin{remark}\label{rem:semibrick}
If a brick $B$ belongs to $\Tors_A(I)$, then $B$ cannot be cogenerated
by $T$. More generally, every nonzero object filtered by a semibrick
inside $\Tors_A(I)$ is excluded from $\Sub T$. This gives a
torsion-lattice formulation of the obstruction which does not depend
on composition factors alone; compare the role of semibricks in
$\tau$-tilting theory developed in \cite{Asai}.
\end{remark}

We next record two general mechanisms which force a module into
$\Sub T$. They become annihilator-detection criteria when combined
with Corollary \ref{cor:sub-equivalence}.

\begin{proposition}
\label{prop:injective-envelope}
Let $X,T\in\modA$. If the injective envelope $E(\soc X)$ belongs to
$\add T$, then $X\in\Sub T$.
\end{proposition}

\begin{proof}
The socle of a finite-length module is essential, so
$E(X)=E(\soc X)$. Hence $X$ embeds in an object of $\add T$.
\end{proof}

For a module $T$, put
\[
 \Sigma_{\mathrm{inj}}(T)=
 \{[S]\in\Sim(A)\mid E(S)\in\add T\}.
\]

\begin{corollary}\label{cor:injective-support}
Let $T$ be self-orthogonal and $\tau$-tilting, and put
$I=\Ann_A(T)$.  If
\[
                \supp\soc I\subseteq\Sigma_{\mathrm{inj}}(T),
\]
then $I=0$ and $T$ is $1$-tilting.
\end{corollary}

\begin{proof}
The support inclusion gives $E(\soc I)\in\add T$, after taking enough
copies to account for multiplicities. Apply Proposition
\ref{prop:injective-envelope} and Corollary
\ref{cor:sub-equivalence}.
\end{proof}

\begin{proposition}\label{prop:socle-lifting}
Let $X,T\in\modA$.  If
\[
        \soc X\in\Sub T
        \quad\text{and}\quad
        \Ext_A^1(X/\soc X,T)=0,
\]
then $X\in\Sub T$.
\end{proposition}

\begin{proof}
Choose a monomorphism $u:\soc X\to T^m$. Applying
$\Hom_A(-,T^m)$ to
\[
 0\longrightarrow\soc X\longrightarrow X\longrightarrow
 X/\soc X\longrightarrow0
\]
shows that $u$ extends to a map $\widetilde u:X\to T^m$. If
$\ker\widetilde u\neq0$, then it contains a simple submodule, which
must lie in $\soc X$. This contradicts the injectivity of $u$. Thus $\widetilde u$ is monomorphic.
\end{proof}

Complete orthogonality also produces dimension shifts along every
submodule of the annihilator.

\begin{proposition}
\label{prop:ann-shift}
Let $T$ be a self-orthogonal $\tau$-tilting module, set
$I=\Ann_A(T)$, and let $L\subseteq I$ be a submodule. Then
\[
                 \Hom_A(I/L,T)=0
\]
and there are natural isomorphisms
\[
 \Ext_A^{n+1}(I/L,T)\simeq\Ext_A^n(L,T)
 \qquad(n\geq0),
\]
where $\Ext_A^0(-,T)=\Hom_A(-,T)$.
\end{proposition}

\begin{proof}
Apply $\Hom_A(-,T)$ to
$0\to L\to I\to I/L\to0$ and use Theorem
\ref{thm:orthogonality}.
\end{proof}

Taking $L=\soc I$ gives
\begin{equation}\label{eq:socle-shift}
        \Ext_A^1(I/\soc I,T)\simeq\Hom_A(\soc I,T).
\end{equation}
Thus a nonzero map $\soc I\to T$ is exactly the obstruction to the
vanishing of $\Ext_A^1(I/\soc I,T)$. In particular, the two
hypotheses in Proposition \ref{prop:socle-lifting} cannot be verified
simultaneously for $X=I\neq0$.

Finally, for the usual Matlis duality $D$, one has
\begin{equation}\label{eq:dual-cogen}
                 I\in\Sub T
       \quad\Longleftrightarrow\quad
                 DI\in\Fac(DT)
\end{equation}
as modules over $A^{\mathrm{op}}$. Thus a filtration proof that
$DI\in\Fac(DT)$ is another possible route to faithfulness. In view
of Theorem \ref{thm:torsion-separation}, such a proof must force the
annihilator to vanish; this explains why a semibrick filtration of
$DI$ cannot be chosen independently of the orthogonality constraints.

 We end the section with the example which separates ideal-top hitting from socle sincerity.

 \begin{example}
\label{ex:hitting-not-sincere}
Let $A=kQ$ for $Q:1\xrightarrow{\alpha}2$, again using right
modules, and take $T=A_A$. The only nonzero nilpotent two-sided ideal
is $J=k\alpha$, and $J\simeq S_2$ as a right module. Hence
\[
                    \mathfrak T(A)=\bigl\{\{[S_2]\}\bigr\}.
\]
Now $\supp\soc(A_A)=\{[S_2]\}$, so this support is an ideal-top
hitting set. It is not sincere, since $S_1$ is absent. Thus the
hypothesis of Corollary \ref{cor:hitting} is strictly weaker than that
of Corollary \ref{cor:socle-sincere}.
\end{example}

\section{Applications to radical square zero algebras}

In this section we give applications of Theorem A and Theorem B to radical square zero algebras, which is a partial answer to the self-orthogonal $\tau$-tilting conjecture, see \cite{Zhang, CLZZ} for details.

For an Artin algebra $\Gamma$, write
$\findim\Gamma$ for the little finitistic dimension computed on
finitely generated right $\Gamma$-modules.

\begin{proposition}[Finitistic-dimension obstruction]
\label{prop:findim-obstruction}
Let $T$ be a self-orthogonal $\tau$-tilting $A$-module and put
$B=\End_A(T)^{\rm op}$. If there exists a nonzero module
$X\in\Fac T$ such that
\[
       \Hom_A(X,T)=0
       \quad\text{and}\quad
       \Ext_A^j(X,T)=0\qquad(j\geq1),
\]
then
\[
                    \findim B=\infty.
\]
\end{proposition}

\begin{proof}
We regard $\Hom_A(-,T)$ as taking values in right $B$-modules via
postcomposition by endomorphisms of $T$.
Put $I=\Ann_A(T)$ and $\overline A=A/I$. Every module in $\Fac T$
is annihilated by $I$, and hence
\[
       \Fac_A T=\Fac_{\overline A}T,
       \qquad
       \Hom_A(U,V)=\Hom_{\overline A}(U,V)
\]
for $U,V\in\Fac T$. The $\overline A$-module $T$ is a classical
$1$-tilting module. Consequently,
\begin{equation}\label{eq:tilting-torsion-class}
        \Fac T=T^{\perp_1}
        \quad\text{inside }\mathsf{mod}\,\overline A.
\end{equation}
Equivalently, with its inherited exact structure, $\Fac T$ has enough
projectives and its projective objects are precisely those in
$\add T$; see \cite[Lemma~2.3]{CLZZ}.

Set $X_0=X$. Inductively, choose a minimal right
$\add T$-approximation $p_m:T_m\twoheadrightarrow X_m$ and put
$X_{m+1}=\ker p_m$.  Such an approximation is epimorphic because
$X_m$ is generated by $T$.  We claim that $X_{m+1}\in\Fac T$.
Indeed, applying $\Hom_{\overline A}(T,-)$ to
\[
0\longrightarrow X_{m+1}\longrightarrow T_m
 \xrightarrow{p_m}X_m\longrightarrow0
\]
and using both the approximation property and
$\Ext^1_{\overline A}(T,T_m)=0$ gives
$\Ext^1_{\overline A}(T,X_{m+1})=0$. The claim follows from
\eqref{eq:tilting-torsion-class}.

Self-orthogonality of $T$ over $A$ and dimension shifting now give
\[
             \Ext_A^j(X_m,T)=0\qquad(j\geq1,\ m\geq0).
\]
Write $H_m=\Hom_A(X_m,T)$ and $P_m=\Hom_A(T_m,T)$. Applying
$\Hom_A(-,T)$ to the preceding short exact sequences yields
\[
             0\longrightarrow H_m\longrightarrow P_m
             \longrightarrow H_{m+1}\longrightarrow0.
\]
Since $H_0=0$, splicing gives, for each $n\geq0$, an exact sequence
of right $B$-modules
\begin{equation}\label{eq:B-projective-resolution}
0\longrightarrow P_0\longrightarrow P_1\longrightarrow\cdots
 \longrightarrow P_n\longrightarrow H_{n+1}\longrightarrow0.
\end{equation}
Each $P_m$ is finitely generated projective. Moreover, the chosen
right $\add T$-approximations are projective covers in the
Krull--Schmidt exact category $\Fac T$. They therefore form a
minimal $\add T$-resolution: its differentials
$T_{m+1}\to T_m$ lie in the radical of $\add T$.
Under the equivalence
\[
       \Hom_A(-,T):(\add T)^{\rm op}\longrightarrow
                    \operatorname{proj}B,
\]
sequence \eqref{eq:B-projective-resolution} is therefore a minimal
projective resolution. Since $X\neq0$, also $T_0\neq0$, so
$P_0\neq0$. Minimality implies that $H_{n+1}\neq0$ and that
\[
       \operatorname{pd}_{B}H_{n+1}=n\qquad(n\geq0).
\]
Thus $B$ has finitely generated modules of arbitrarily large finite
projective dimension.
\end{proof}

\begin{corollary}
\label{cor:finite-findim-detection}
Let $T$ be a self-orthogonal $\tau$-tilting $A$-module and
$B=\End_A(T)^{\rm op}$. If $\findim B<\infty$, then
\[
 \Fac T\cap{}^{\perp_{\geq0}}T=\{0\}.
\]
In particular, if $I=\Ann_A(T)$ belongs to $\Fac T$, then $I=0$ and
$T$ is $1$-tilting.
\end{corollary}

\begin{proof}
The first assertion is the contrapositive of Proposition
\ref{prop:findim-obstruction}. For the second, use Theorem
\ref{thm:orthogonality}, which gives $I\in{}^{\perp_{\geq0}}T$.
\end{proof}

We now obtain an application to radical square zero algebras, which is a partial answer to the self-orthogonal $\tau$-tilting conjecture, see \cite{CLZZ,Zhang} for details.

\begin{theorem}
\label{thm:radical-square-zero}
Let $A$ be a radical square zero Artin algebra.
Then every self-orthogonal $\tau$-tilting $A$-module is a classical
$1$-tilting module.
\end{theorem}

\begin{proof}
Let $T$ be self-orthogonal and $\tau$-tilting, and set
$I=\Ann_A(T)$ and $B=\End_A(T)^{\rm op}$. Suppose that $I\neq0$.
Since $I$ is
nilpotent, $I\subseteq\rad A$.  Hence
\[
                         I\rad A=0,
\]
so $I_A$ is semisimple. Let $S$ be a simple direct summand of $I$.
By Theorem \ref{thm:orthogonality},
\[
       \Hom_A(S,T)=0,
       \qquad
       \Ext_A^j(S,T)=0\quad(j\geq1).
\]

A $\tau$-tilting module is sincere. Thus $S$ occurs as a composition
factor of $T$. Since $\rad^2 A=0$, one has
$\rad T=T\rad A\subseteq\soc T$. If $S$ did not occur in $\topm T$,
then its occurrence in $T$ would lie in $\rad T$, yielding a nonzero
map $S\to T$, contrary to $\Hom_A(S,T)=0$. Therefore
$S\in\supp\topm T$. Hence there is an epimorphism
$T\twoheadrightarrow S$, and so $S\in\Fac T$. Repeating this for
every simple summand of the semisimple module $I$ gives
\[
                              I\in\Fac T.
\]
Proposition \ref{prop:findim-obstruction}, applied to $X=I$, now gives
\[
                      \findim B=\infty.
\]

On the other hand, put $\overline A=A/I$. The module $T$ is a
classical $1$-tilting $\overline A$-module, and
$B\simeq\End_{\overline A}(T)^{\rm op}$. The algebras
$\End_{\overline A}(T)$ and $\overline A$ are derived equivalent;
hence $B$ and $\overline A^{\rm op}$ are derived equivalent.
Moreover, $\rad^2\overline A=0$, so
$\findim\overline A^{\rm op}<\infty$ by the radical-cube-zero theorem
of Green and Huisgen-Zimmermann \cite[Theorem 16]{GreenHuisgen}. Finiteness of the
little finitistic dimension is preserved under derived equivalence
by Pan-Xi \cite[Theorem 1.1]{PanXi}. Therefore
\[
                     \findim B<\infty,
\]
a contradiction. Thus $I=0$, and the faithful $\tau$-tilting
criterion yields that $T$ is $1$-tilting.
\end{proof}

One may ask whether there is a similar result on algebras with radical cube zero. In general, we don't know. However, we have the following result.

\begin{proposition}
\label{prop:radical-cube-zero-reduction}
Let $A$ be an Artin algebra with $\rad^3 A=0$, and let
$T$ be a self-orthogonal $\tau$-tilting $A$-module. Put
$I=\Ann_A(T)$ and $B=\End_A(T)^{\rm op}$. Then
\[
                    \findim B<\infty.
\]
Consequently, if $I\neq0$, then necessarily
\[
                              I\notin\Fac T.
\]
More generally,
\[
                  \Fac T\cap{}^{\perp_{\geq0}}T=\{0\}.
\]
\end{proposition}

\begin{proof}
The quotient $\overline A=A/I$ again satisfies
$\rad^3\overline A=0$. Green and Huisgen-Zimmerman's theorem \cite[Theorem 16]{GreenHuisgen} gives
$\findim\overline A^{\rm op}<\infty$. Since $T$ is a
$1$-tilting $\overline A$-module, $B$ is derived equivalent to
$\overline A^{\rm op}$, and Pan-Xi's theorem \cite[Theorem 1.1]{PanXi} gives
$\findim B<\infty$. The remaining assertions
follow from Corollary \ref{cor:finite-findim-detection} and Theorem
\ref{thm:orthogonality}.
\end{proof}

\begin{remark}
\label{rem:radical-cube-zero-bottleneck}
Proposition \ref{prop:radical-cube-zero-reduction} isolates the missing
step for $\rad^3 A=0$: one has to force the annihilator into
$\Fac T$. The argument for Theorem \ref{thm:radical-square-zero}
does not extend formally. Indeed, when $\rad^3 A=0$, a simple
composition factor which is absent from $\soc T$ may occur in the
middle Loewy layer $T\rad A/T\rad^2 A$ rather than in $\topm T$.
Thus the radical-square-zero proof fails for a precise structural
reason, not because of the finitistic-dimension step. Any proof of the
radical-cube-zero case may therefore be reduced to controlling this
middle-layer escape.
\end{remark}

\section{Applications to triangular matrix algebras}

In this section we prove the projective-support criterion and apply it to triangular
matrix algebras and one-point extensions. We also give the explicit
$51$-dimensional family.

We now turn the ideal-top criterion into conditions on the algebra
alone. Assume for the moment that $A$ is basic, with simple modules
$S_1,\dots,S_n$ and indecomposable projectives $P_i=e_iA$.
The conditions below are unchanged under Morita equivalence when they
are formulated for a basic representative.

\begin{lemma}
\label{lem:projective-support-transversal}
If $M$ is sincere, then
\[
            \supp\soc M\cap\supp P_i\neq\varnothing
            \qquad(1\leq i\leq n).
\]
\end{lemma}

\begin{proof}
Sincerity gives $Me_i\neq0$, or equivalently a nonzero map
$P_i\to M$.  Its image is a nonzero quotient of $P_i$. Any simple
submodule of this image belongs to $\soc M$ and is a composition
factor of $P_i$.
\end{proof}

Call a subset $U\subseteq\Sim(A)$ \emph{projective-support
absorbing} if $\supp P_i\subseteq U$ for some $i$.

\begin{theorem}[Projective-support absorption]
\label{thm:projective-support-absorption}
Assume that $\supp\topm(J_A)$ is projective-support absorbing for
every nonzero nilpotent two-sided ideal $J$ of $A$. Then every
self-orthogonal $\tau$-tilting $A$-module is $1$-tilting.
\end{theorem}

\begin{proof}
Let $T$ be self-orthogonal and $\tau$-tilting. For every nonzero
nilpotent $J$, choose $i$ such that
$\supp P_i\subseteq\supp\topm J$. Since $T$ is sincere, Lemma
\ref{lem:projective-support-transversal} gives
\[
 \varnothing\neq\supp\soc T\cap\supp P_i
 \subseteq\supp\soc T\cap\supp\topm J.
\]
Thus $\supp\soc T$ is an ideal-top hitting set, and Corollary
\ref{cor:hitting} applies.
\end{proof}

Put
\[
 \mathcal H(A)=
 \{i\mid\supp P_i=\{[S_i]\}\}.
\]
The modules $P_i$ with $i\in\mathcal H(A)$ need not be simple; they
may have arbitrary Loewy length, but all their composition factors
are isomorphic to $S_i$.

\begin{corollary}
\label{cor:homogeneous-radical}
Suppose that
\begin{equation}\label{eq:homogeneous-radical}
 \supp(\rad A)_A\subseteq
 \{[S_i]\mid i\in\mathcal H(A)\}.
\end{equation}
Then every self-orthogonal $\tau$-tilting $A$-module is
$1$-tilting.
\end{corollary}

\begin{proof}
Every nilpotent ideal $J$ is contained in $\rad A$. If
$0\neq J$, choose $[S_i]\in\supp\topm J$. Condition
\eqref{eq:homogeneous-radical} gives
$\supp P_i=\{[S_i]\}\subseteq\supp\topm J$. Apply Theorem
\ref{thm:projective-support-absorption}.
\end{proof}

Condition \eqref{eq:homogeneous-radical} has a transparent quiver
form. For a bound quiver algebra, it holds whenever every vertex
which occurs as the target of an arrow emits no arrow to a different
vertex; loops at such a vertex are allowed. Thus, after deleting
loops, all arrows run from source vertices to terminal vertices, but
the local algebras carried by the terminal vertices may have
arbitrary Loewy length and representation type.

\begin{corollary}[Semisimple-source triangular algebras]
\label{cor:triangular-family}
Assume that $A$ is Morita equivalent to
\[
 A'=\begin{pmatrix}
       S&M\\
       0&C_1\times\cdots\times C_s
    \end{pmatrix},
\]
where $S$ is a semisimple Artin algebra, each $C_j$ is a local Artin
algebra, and ${}_SM_{C_1\times\cdots\times C_s}$ has finite length on
both sides. Then every
self-orthogonal $\tau$-tilting $A$-module is $1$-tilting.
\end{corollary}

\begin{proof}
Passing to a basic representative, we may assume that $S$ is a product
of division rings. The radical of $A'$ is
\[
 \rad A'=\begin{pmatrix}
            0&M\\
            0&\rad(C_1\times\cdots\times C_s)
          \end{pmatrix}.
\]
Every simple composition factor of this right module is the unique
simple module of one of the local algebras $C_j$. The corresponding
indecomposable projective $A'$-module is supported only at that
vertex. Hence \eqref{eq:homogeneous-radical} holds, and Corollary
\ref{cor:homogeneous-radical} applies.
\end{proof}

Here and below, by a source one-point extension we mean an algebra of
the displayed form. In particular, every such extension
\[
                    \begin{pmatrix}k&M\\0&C\end{pmatrix}
\]
of an arbitrary local algebra $C$ has the asserted property; no
projectivity assumption on $M$ is required. Support $\tau$-tilting
modules over one-point extensions have been studied in
\cite{GaoXie,Suarez}, while $\tau$-tilting finiteness for triangular
matrix algebras was investigated in \cite{AiharaHonma}. Those results
concern construction or finiteness; the conclusion here instead
concerns automatic faithfulness under self-orthogonality.

\begin{example}[A $51$-dimensional example]
\label{ex:triangular-51}
Let $k$ be an infinite field and set
\[
 C_{m,n}=k[x,y]/(x^m,y^n),\qquad
 A_{m,n}=\begin{pmatrix}k&C_{m,n}\\0&C_{m,n}\end{pmatrix}
 \quad(m,n\geq2).
\]
Then $\dim_k A_{m,n}=2mn+1$, and Corollary
\ref{cor:triangular-family} shows that every self-orthogonal
$\tau$-tilting $A_{m,n}$-module is $1$-tilting.  
Taking $m=n=5$ gives a $51$-dimensional example $A_{5,5}$. Then
\vspace{0.2cm}
\begin{enumerate}
\item $A_{5,5}$ is given by the quiver
with one arrow $\alpha:1\to2$, two loops $x,y$ at vertex $2$, and
relations
\[
                       x^5=0=y^5,\qquad xy=yx.
\]

\item $C=C_{5,5}$ and $A_{5,5}$ is representation infinite.  

 Let
\[
          F:\mathop{\rm mod}\nolimits\text{-}C
          \longrightarrow\mathop{\rm mod}\nolimits\text{-}A_{5,5},
          \qquad X\longmapsto (0,X),
\]
be the embedding at the second vertex. It is exact and fully faithful
and sends projective $C$-modules to projective $A_{5,5}$-modules. The
algebra $C$ maps onto $R=k[x,y]/(x,y)^2$. For each $\lambda\in k$, let
$X_\lambda$ be the two-dimensional $R$-module on which
\[
             x\longmapsto N=\begin{pmatrix}0&1\\0&0\end{pmatrix},
             \qquad y\longmapsto\lambda N.
\]
The endomorphism ring of $X_\lambda$ is the local algebra $k[N]$, so
$X_\lambda$ is indecomposable. An isomorphism
$X_\lambda\simeq X_\mu$ would commute with $N$ and force
$\lambda=\mu$. Consequently, $C$ and $A_{5,5}$ are
representation-infinite.

\item $A_{5,5}$ is CM-infinite of infinite global dimension.

 The algebra $C$ is local symmetric. Hence every finitely generated
$C$-module is Gorenstein-projective. If $P^\bullet$ is a totally
acyclic complex of projective $C$-modules, then $F(P^\bullet)$ is a
complex of projective $A_{5,5}$-modules and, for either indecomposable
projective $A_{5,5}$-module $Q$, one has a natural identification
\[
 \Hom_{A_{5,5}}(F(P^\bullet),Q)
       \simeq \Hom_C(P^\bullet,C).
\]
Thus $F(X)$ is Gorenstein-projective for every $C$-module $X$.  The
modules $F(X_\lambda)$ are pairwise nonisomorphic and indecomposable,
so $A_{5,5}$ is CM-infinite. The simple $C$-module is nonprojective;
its image under $F$ is therefore a nonprojective
Gorenstein-projective $A_{5,5}$-module and has infinite projective
dimension. Hence $\operatorname{gl.dim}A_{5,5}=\infty$.

\item $A_{5,5}$ is neither gentle nor self-injective.

Since the simple $S_1\simeq D(A_{5,5}e_1)$ at the source
vertex is injective, but its projective cover has the nonzero
radical $\alpha C$, then it is not self-injective. It is not gentle either. Indeed, at vertex $2$ put
$V=\rad C/\rad^2C$. For every $0\neq u\in V$, multiplication induces
an injective map
\[
                  V\longrightarrow\rad^2C/\rad^3C,
                  \qquad v\longmapsto uv.
\]
Here injectivity follows from the linear independence of
$x^2,xy,y^2$ modulo $\rad^3C$. Thus every choice of arrow basis has
an arrow with two nonzero successors, contradicting the
special-biserial condition required of a gentle algebra.
\end{enumerate}
\end{example}

\begin{remark}[Earlier special cases]
\label{rem:earlier-special-cases}
The simple-projective ideal criterion, the case in which $\rad A_A$
is semisimple projective, and the one-simple-block case all follow
from Theorem \ref{thm:projective-support-absorption} or Corollary
\ref{cor:homogeneous-radical}. Hence finite-dimensional commutative
algebras and arbitrary local Artin algebras are retained as special
cases, but they are not the principal applications here.
\end{remark}

We next specialize the criterion to symmetric algebras. Assume that
$A$ is basic symmetric. Choose simple modules
$S_1,\dots,S_n$ and indecomposable projectives $P_i=e_iA$ so that
$\soc P_i\simeq S_i$. For a module $T$, define its projective
footprint by
\[
                       \Pi(T)=\{[S_i]\mid P_i\in\add T\}.
\]
Since $\Pi(T)\subseteq\supp\soc T$, Corollary \ref{cor:hitting} gives
the following finite criterion.

\begin{corollary}
\label{cor:symmetric-footprint}
Let $A$ be a basic symmetric algebra and $T$ a self-orthogonal
$\tau$-tilting $A$-module. If $\Pi(T)$ is an ideal-top hitting set,
then the basic representative of $\add T$ is isomorphic to $A$.
\end{corollary}

\begin{proof}
Corollary \ref{cor:hitting} implies that $T$ is $1$-tilting. Over a
self-injective algebra, every module of finite projective dimension
is projective. Hence $T$ is a projective generator, and its basic
form is $A$.
\end{proof}

This criterion is deliberately stated as a
consequence rather than the main result: depending on the bimodule
socle of a symmetric algebra, it may be much stronger than the
socle-support condition of Corollary \ref{cor:hitting}.

We give a consolidated list of restrictions. It can be used
as an initial filter in a computational search for a counterexample
to the self-orthogonal $\tau$-tilting conjecture.

\begin{proposition}
\label{prop:counterexample-profile}
Suppose that $T$ is a self-orthogonal $\tau$-tilting $A$-module which
is not $1$-tilting, and put $I=\Ann_A(T)$. Then:
\begin{enumerate}[\rm(i)]
\item $I$ is a nonzero nilpotent two-sided ideal;
\item $\Hom_A(I,T)=0$ and $\Ext_A^n(I,T)=0$ for every $n\geq1$;
\item $\Hom_A(\Tors_A(I),\Sub T)=0$ and
      $\Tors_A(I)\cap\Sub T=\{0\}$;
\item $\supp\topm I\cap\supp\soc T=\varnothing$;
\item for every submodule $L\subseteq I$ and every $n\geq0$,
\[
          \Ext_A^{n+1}(I/L,T)\simeq\Ext_A^n(L,T).
\]
\end{enumerate}
If $A$ is basic symmetric, then $\supp\topm I$ also avoids the
projective footprint $\Pi(T)$.
\end{proposition}

\begin{proof}
Part (i) follows from sincerity of $\tau$-tilting modules and the
failure of faithfulness.  Parts (ii)--(v) are Theorems
\ref{thm:orthogonality} and \ref{thm:torsion-separation}, Corollary
\ref{cor:top-socle}, and Proposition \ref{prop:ann-shift},
respectively. In the symmetric case,
$\Pi(T)\subseteq\supp\soc T$.
\end{proof}

We end the paper with the following concrete problems suggested by the separation theorem.

\begin{question}
For which bound quiver algebras does every nonzero nilpotent ideal
$J$ satisfy
$\Tors_A(J)\cap\Sub(A/J)\neq\{0\}$? Can this property be recognized
from the path poset and the minimal relations, without enumerating
two-sided ideals?
\end{question}

\begin{question}
Can a nonzero nilpotent ideal and a $\tau$-tilting module satisfy all
conditions in Proposition \ref{prop:counterexample-profile}
simultaneously? A negative answer in a stable class of algebras would
give a new route from self-orthogonality to tilting.
\end{question}

{\bf Acknowledgements} The research of the author is supported by NSFC(Nos. 12171207, 12371038). The author wants to thank Professors Hongxing Chen, Xiao-Wu Chen, Zhi-Wei Li, Shengyong Pan and Zhibing Zhao for useful suggestions.

\end{document}